\documentclass{amsart}
\usepackage{amssymb,amsthm}
\usepackage{amsfonts}
\usepackage{enumerate}
\usepackage{amsmath}
\usepackage{amsrefs}
\usepackage{graphicx}
\usepackage[utf8]{inputenc}
\usepackage[english]{babel}
\usepackage{color}
\usepackage{hyperref}
\usepackage{pgf}
\usepackage{colortbl,xcolor}
\usepackage{multicol}
\usepackage{pgf}
\usepackage{pagecolor}

\theoremstyle{plain}

\newtheorem{definition}{Definition}[section]
\newtheorem{example}{Example}[section]

\newtheorem{lemma}{Lemma}[section]

\newtheorem{remark}{Remark}[section]

\newtheorem{theorem}{Theorem}[section]
\numberwithin{equation}{section}

\definecolor{lightgrey}{cmyk}{0,0,0,0.30}
\definecolor{darkgrey}{cmyk}{0,0,0,0.70}
\definecolor{purple}{cmyk}{0.45,0.86,0,0}
\definecolor{darkblue}{cmyk}{1.7,0.7,0.3,0}
\definecolor{lightblue}{cmyk}{0.6,0.2,0.1,0}
\definecolor{midblue}{cmyk}{1.1,0.3,0.2,0}

\begin{document}	
\pagecolor{yellow!10!}
	\title[Formula for the local zeta function of a Laurent polynomial]{An explicit formula for the local zeta function of a Laurent polynomial}
	\author{Edwin Le\'{o}n-Cardenal}
	\address{CONACyT Research Fellow -- Centro de Investigaci\'{o}n en Matemáticas \\
		Unidad Zacatecas\\
		Avenida Universidad \#222, Fracc. 
		La Loma, Zacatecas, ZAC. 98068\\
		M\'{e}xico.}
	\email{edwin.leon@cimat.mx}
	\subjclass[2000]{Primary 11S40, 14G10; Secondary 52B20}
	\keywords{Laurent polynomials, Igusa's zeta function, explicit formulae, Newton polytopes, non-degeneracy conditions}
\begin{abstract}
In a recent paper Z{\'u}{\~n}iga-Galindo and the author begun the study of the local zeta functions for Laurent polynomials. In this work we continue this study by giving a very explicit formula for the local zeta function associated to a Laurent polynomial $f$ over a $p-$adic field, when $f$ is weakly non-degenerate with respect to the Newton polytope of $f$ at infinity. 
\end{abstract}
\maketitle
\section{Introduction}

Local zeta functions were introduced in the 60's by Israel Gel'fand and Andr\'e Weil. 
Gel'fand study this functions over $\mathbb{R}$ with the aim of showing 
the existence of fundamental solutions to certain partial differential equations with constant coefficients. Meanwhile Weil studied $p-$adic local zeta functions in order to generalize some results of Siegel about quadratic forms, e.g. the Siegel-Poisson formula, see \cite{De}. Significant contributions have since been made, specially in the last three decades, see e.g. \cites{De,IBook,Nic} and the references therein.

One of the most powerful tools in the characteristic zero wing of the theory is the resolution of singularities. In particular, after the pioneering work of Varchenko \cite{Var}, Newton polyhedra techniques have been extensively employed to study local zeta functions as well as their connections with oscillatory integrals, see e.g. \cites{De,LeVeZu} and the references therein for the Archimedean case, and \cites{DeHoo,VeZu,Zu}, among others, in the non-Archimedean case, including the positive characteristic case. In \cites{LeZu,LeZu2} we began the study of local zeta functions for a Laurent polynomial $f$ over a $p-$adic field. There we introduce a Newton polytope at infinity $\Gamma_\infty$ associated to $f$, together with a non-degeneracy condition in order to show the existence of a meromorphic continuation to the whole complex plane of the local zeta functions for $f$. We also obtain asymptotic expansions for $p-$adic oscillatory integrals attached to Laurent polynomials and give bounds for the size of `tubular neighborhoods' attached to the polynomials. Our work in \cite{LeZu} is closely related with the recent paper \cite{VeZu2}, where the authors study the local zeta functions for meromorphic functions.

The main tool used in \cite{LeZu} for the meromorphic continuation of the local zeta functions is a variation of toric resolution of singularities.  In the classical case of a polynomial function $f$,  one uses the Newton polyhedron of $f$ to construct a 
conical decomposition of the first orthant of $\mathbb{R}^n$ into simple cones, this decomposition is called a fan. Then one construct a toric manifold and a map of it into $\mathbb{R}^n$, which together resolve the singularities of almost all the critical points of $f$. This strategy was carried out successfully in \cite{LeZu} for the case of Laurent polynomials. However, there is a detail that should be pointed out about our construction. In general, the set of generators of the simple fan may contain extra rays coming from the intersection of
$\mathbb{R}_+^n$ with the cones in the original fan, and these extra rays could lead to superfluous candidate poles for the local zeta function, see Example \ref{Ex 2} and the discussion before Example \ref{ex. 1}. 

In this paper we justified the aforementioned fact and shows that it may happen only in dimensions above 2. Further work in the description of this new set of rays could be of some interest. In the second part of this work we go back to the referred construction of  \cite{LeZu}, in particular we review and slightly refine the definition of  the conical partition subordinated to the Newton polytope of a non--degenerate Laurent polynomial $f$. Our goal is to give an explicit formula for 
the local zeta function attached to a character of the group of units of the local ring of $K$ and a non--degenerate Laurent polynomial $f$, see Section \ref{Sect_EF}.  Furthermore, we provide several examples that we hope may shed more light on our approach in \cite{LeZu} and here, to the study of this new type of local zeta functions. 

\section{Newton Polytopes and Non-degeneracy Conditions}\label{Sec1}
\subsection{Newton Polytopes}

We take $\mathbb{R}_{+}:=\{x\in\mathbb{R} \mathbf{ ; }\ x\geqslant0\}$. If 
$\left\langle \cdot,\cdot\right\rangle $ denotes the usual inner product of
$\mathbb{R}^{n}$, we identify  the dual space of $\mathbb{R}^{n}$ with
$\mathbb{R}^{n}$ itself by means of it.

Let $K$ be a local field of characteristic zero. Let
\[
f(x_1,\ldots,x_n)=\sum_{m\in S}c_{m}x^m\in K[x_1^{\pm 1},\ldots,x_n^{\pm 1}],
\]
be a non-constant Laurent polynomial with $S$ a finite subset of $\mathbb{Z}^n$, and $c_{m}\in K\setminus\{0\}$ for all $m\in S$. The set $S$ is called the support
of $f$.  We define the \textit{Newton polytope} $\Gamma_{\infty}\left(  f\right)  :=\Gamma_{\infty}$ \textit{\ of
}$f$\textit{\ at infinity} as the convex hull of $S$ in $\mathbb{R}^{n}$. 
From now on, we assume that $\dim\Gamma_{\infty}=n$.

Let $H$ be the hyperplane $\left\{  x\in\mathbb{R}^{n}\ ;\ \left\langle
a,x\right\rangle =b\right\}  $. Then $H$ determines two closed half-spaces:
\[
H^{+}:=\left\{  x\in\mathbb{R}^{n}\mathbf{;}\left\langle a,x\right\rangle \geq
b\right\}
\quad\text{and}\quad
H^{-}:=\left\{  x\in\mathbb{R}^{n}\mathbf{;}\left\langle a,x\right\rangle \leq
b\right\}  .
\]
We say that $H$ is \textit{a supporting hyperplane} of $\Gamma_{\infty}$, if
$\Gamma_{\infty}\cap H\neq\emptyset$ and $\Gamma_{\infty}\subset H^+$ or $\Gamma_{\infty}\subset H^-$. A \textit{face }of $\Gamma
_{\infty}$ is the intersection of the polytope with a supporting hyperplane. Faces of dimension $0, 1$, and $n-1$ are called vertices, edges and facets, respectively. 
We denote by $Vert(\Gamma_{\infty})$ the set of vertices of \ $\Gamma_{\infty}$. 

Given $a\in\mathbb{R}^{n}$, we define
\begin{equation*}
d(a)=\inf\left\{  \left\langle a,x\right\rangle \mathbf{;}\ x\in\Gamma_{\infty
}\right\}  .
\end{equation*}
In fact $d(a)=\min\left\{  \left\langle a,x\right\rangle \mathbf{;}\ x\in Vert(\Gamma
_{\infty})\right\}  ,$ furthermore $d(a)=\left\langle a,x_{0}\right\rangle $ for some $x_{0}\in
Vert(\Gamma_{\infty})$.

Now, given a supporting hyperplane $H$ of $\Gamma_{\infty}$ containing a facet of $\Gamma_{\infty}$, there exists a
unique vector $a\in \mathbb{Z}^n\setminus\{0\}$ which is orthogonal to $H$ and is directed into the polytope, such a vector is called the \textit{inward} normal to $H$. A vector $a=(a_{1},\ldots,a_{n})\in\mathbb{Z}^{n}$ is called primitive if
$g.c.d.(a_{1},\ldots,a_{n})=1$, so when the vector $a$ is chosen to be primitive, it turns out that every facet of $\Gamma_{\infty}$ has a unique primitive inward vector; the set of such vectors is denoted by  $\mathfrak{D}(\Gamma_\infty)$.

\subsection{Conical Subdivisions of $\mathbf{\mathbb{R}_+^n}$} \label{Conical}

We present here the main results about the conical partition that we use in \cite{LeZu}, where the reader may find more details.

If $a\in\mathbb{R}^{n}$,  \textit{the first meet locus} of $a$ is defined as
\[
F(a)=\left\{  x\in\Gamma_{\infty} \mathbf{ ; }\left\langle a,x\right\rangle
=d\left(  a\right)  \right\}  .
\]
Note that $F(a)$\ is a face of $\Gamma_{\infty}$, and that $F(0)=\Gamma
_{\infty}$. With this notion at hand we define an equivalence relation on $\mathbb{R}^{n}$\ by taking
\[
a\sim a^{\prime}\Longleftrightarrow F\left(  a\right)  =F\left(  a^{\prime
}\right).
\]
Now, given $a_1,\ldots,a_k\in\mathbb{R}^{n}\ (k\leq n)$ we call 
\begin{equation}\label{cone1}
\Delta=\left\{  \lambda
_{1}a_1+\cdots+\lambda_{k}a_k \mathbf{ ; }\ \lambda_{i}\in\mathbb{R}, \lambda_{i}>0\right\}
\end{equation} the \textit{cone strictly spanned} by $a_1,\ldots,a_k$. When the generators $a_1,\ldots,a_k$ are linearly independent over $\mathbb{R}$, the cone is called \textit{simplicial} and when $a_1,\ldots,a_k\in\mathbb{Z}^n$ the cone is called \textit{rational}. If $\{a_1,\ldots,a_k\}$ is a subset of a basis of the $\mathbb{Z}-$module $\mathbb{Z}^n$, we call $\Delta$ a \textit{simple cone}.

In order to describe the equivalence classes of $\sim$ we define \textit{the cone
	associated to} $\tau$, a given face of $\Gamma_{\infty}$, as
\begin{equation}\label{cone2}
\Delta_{\tau}=\left\{  a\in\mathbb{R}^{n}\mathbf{;}\ F\left(  a\right)
=\tau\right\}  .
\end{equation}
Note that $\Delta_{\Gamma_\infty}=\{0\}$. The other equivalence classes are described in the next Lemma, which also provides the relation between 
(\ref{cone1}) and (\ref{cone2}).

\begin{lemma}\label{lemma1}
\begin{enumerate}[(i)]
Let $\tau$ be a face of $\Gamma_{\infty}$, $\tau\neq\Gamma_\infty$, then
\item the topological closure $\overline{\Delta}_{\tau}$ of
$\Delta_{\tau}$ is a rational polyhedral cone and
\[\overline{\Delta}_{\tau}=\left\{  a\in\mathbb{R}^{n}\mathbf{;}F\left(a\right)  \supset\tau\right\}  .\]
\item  Let $\gamma_{1},\ldots,\gamma_{k}$ be the facets of $\Gamma_{\infty}$ containing
$\tau$. Let $a_{1},\ldots,a_{k}\in\mathbb{Z}^{n}\setminus\left\{
0\right\}$ be the unique primitive orthogonal inward vectors to $\gamma_{1}%
,\ldots,\gamma_{k}$ respectively. Then
\begin{gather*}
\Delta_{\tau}=\left\{\lambda_{1}a_{1}+\cdots+\lambda_ka_k\ ;\ \lambda_{i}\in\mathbb{R},\text{ and }\lambda_{i}>0\right\},\\
\text{ and }\quad\overline{\Delta}_{\tau}=\left\{\lambda_{1}a_{1}+\cdots+\lambda_ka_k\ ;\ \lambda_{i}\in\mathbb{R},\text{ and }\lambda_{i}\geq0\right\}.
\end{gather*}
\item $\dim\Delta_{\tau}=\dim\overline{\Delta}_{\tau}=n-\dim\tau.$
\end{enumerate}
\end{lemma}

We recall that a \textit{fan} $\mathcal{L}$ is a finite collection of rational polyhedral
cones $\{\Lambda_i; i\in I\}$ in $\mathbb{R}^n$ such that: (i) if $\Lambda_i\in \mathcal{L}$ and $\Lambda$ is a face of $\Lambda_i$, then $\Lambda\in \mathcal{L}$; (ii) if $\Lambda_1, \Lambda_2\in \mathcal{L}$, then $\Lambda_1\cap\Lambda_2$ is a face of $\Lambda_1$ and $\Lambda_2$. The \textit{support} of $\mathcal{L}$ is $\cup_{i\in I}\Lambda_i$. A fan $\mathcal{L}$ is called \textit{simplicial} (resp. \textit{simple} ) if all its cones are
simplicial (resp. simple). A fan $\mathcal{L}$ is called \textit{subordinated to} $\Gamma_\infty$, if every cone in $\mathcal{L}$ is contained in an equivalence class of $\sim$. We denote by $gen(\mathcal{L})$, the set of all generators of the cones in $\mathcal{L}$.

Let $\{e_1,\ldots,e_n\}$ denote the canonical basis of $\mathbb{R}^n$ and denote by $\Delta_i$ the cone strictly spanned by $e_i$. Note that if $\Delta_{\tau}\cap\mathbb{R}_{+}^{n}\neq\emptyset$, then $\Delta_{\tau}\cap\mathbb{R}_{+}^{n}$ is a cone of the form (\ref{cone1}), therefore we have:
\begin{lemma}\label{lemma2}
\begin{enumerate}
\item $\{\Delta_\tau\ ;\ \tau \text{ is a face of } \Gamma_\infty\}:=\{\Delta_\tau\}$ is a conical partition of $\mathbb{R}^n$ into open cones. 
\item If $Int(\Delta)$ denotes the topological interior of a cone $\Delta$, then $$\mathcal{C}:=\{Int(\Delta_{\tau}\cap\mathbb{R}_{+}^{n})\ \mathbf{;}\ \Delta_{\tau}\cap\mathbb{R}_{+}^{n}\neq\emptyset\}\bigcup \cup_{i=1}^n\Delta_i,$$ forms a conical partition of $\mathbb{R}_{+}^{n}\setminus\{0\}$ into open cones.
\item $\{\overline{\Delta_\tau}\}$ (resp. $\overline{\mathcal{C}}:=\{\overline{\Delta_\tau}\}\bigcup \cup_{i=1}^n\overline{\Delta}_i$) is a fan
subordinated to $\Gamma_\infty$ with support $\mathbb{R}^{n}$ (resp. $\mathbb{R}_{+}^{n}$). 
\end{enumerate}
\end{lemma}
Each cone in $\mathcal{C}$ can be partitioned into a finite number of simplicial cones. By adding new rays,
each simplicial cone can be partitioned further into a finite number of simple cones. In this way we may find simple fans containing $\{\overline{\Delta_\tau}\}$ (resp. $\overline{\mathcal{C}}$) and subordinated to $\Gamma_\infty$. From now on, we fix a simple fan $\mathcal{F}$ subordinated to $\Gamma_\infty$ with support $\mathbb{R}_{+}^{n}$. We set $\mathcal{F}_0$ to be the cone $\mathbb{R}_{+}^{n}$ and its faces, and we will say that $\mathcal{F}$ is \textit{trivial} if $\mathcal{F}=\mathcal{F}_0$. 
 
\begin{example}\label{ex. 1}
Given arbitrary $u,v \in \mathbb{N}$, set $f(x,y) = (y^{-1}+x)^u+y^v\in K[x^{\pm1},y^{\pm1}]$. 
\begin{figure}[ht]
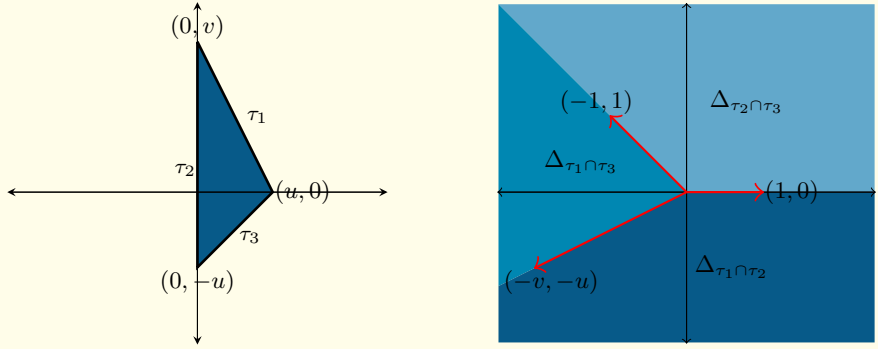
\label{Fig1}	
\begin{multicols}{2}
\begin{pgfpicture}{-2.5cm}{-2cm}{2.5cm}{2.5cm}
        \color{darkblue}
        \pgfmoveto{\pgfxy(0,2)}
        \pgflineto{\pgfxy(0,-1)}
        \pgflineto{\pgfxy(1,0)}
        \pgflineto{\pgfxy(0,2)}
        \pgfclosepath
        \pgffill
        \pgfsetlinewidth{1pt}
        \color{black}
        \pgfmoveto{\pgfxy(0,2)}
        \pgflineto{\pgfxy(0,-1)}
        \pgflineto{\pgfxy(1,0)}
        \pgflineto{\pgfxy(0,2)}
        \pgfstroke
        \pgfsetlinewidth{0.4pt}
        \pgfsetendarrow{\pgfarrowsingle}
        \pgfxyline(-2.5,0)(2.5,0)
        \pgfxyline(0,-2)(0,2.5)
        \pgfxyline(2.5,0)(-2.5,0)
        \pgfxyline(0,2.5)(0,-2)
                
        \pgfputat{\pgfxy(0.8,1)}{\pgfbox[center,center]{\footnotesize $\tau_1$}}
        \pgfputat{\pgfxy(-0.3,0.3)}{\pgfbox[left,center]{\footnotesize $\tau_2$}}
        \pgfputat{\pgfxy(0.7,-0.6)}{\pgfbox[center,center]{\footnotesize $\tau_3$}}
        \pgfputat{\pgfxy(0,2.2)}{\pgfbox[center,center]{\small{$(0,v)$}}}
        \pgfputat{\pgfxy(1.4,0)}{\pgfbox[center,center]{\small{$(u,0)$}}}
        \pgfputat{\pgfxy(0,-1.2)}{\pgfbox[center,center]{\small{$(0,-u)$}}}
        \end{pgfpicture}

\columnbreak
\begin{pgfpicture}{-2.5cm}{-2cm}{2.5cm}{2.5cm}
        \color{lightblue}
        \pgfmoveto{\pgfxy(2.5,0)}
        \pgflineto{\pgfxy(2.5,2.5)}
        \pgflineto{\pgfxy(-2.5,2.5)}
        \pgflineto{\pgfxy(0,0)}
        \pgflineto{\pgfxy(2.5,0)}
        \pgfclosepath
        \pgffill
        \color{midblue}
        \pgfmoveto{\pgfxy(-2.5,2.5)}
        \pgflineto{\pgfxy(-2.5,-1.25)}
        \pgflineto{\pgfxy(0,0)}
        \pgflineto{\pgfxy(-2.5,2.5)}
        \pgfclosepath
        \pgffill
        \color{darkblue}
        \pgfmoveto{\pgfxy(2.5,0)}
        \pgflineto{\pgfxy(2.5,-2)}
        \pgflineto{\pgfxy(-2.5,-2)}
        \pgflineto{\pgfxy(-2.5,-1.25)}
        \pgflineto{\pgfxy(0,0)}
        \pgflineto{\pgfxy(2.5,0)}
        \pgfclosepath
        \pgffill
        \color{black}
        \pgfsetlinewidth{0.4pt}
        \pgfsetendarrow{\pgfarrowto}
        \pgfxyline(-2.5,0)(2.5,0)
        \pgfxyline(0,-2)(0,2.5)
        \pgfxyline(2.5,0)(-2.5,0)
        \pgfxyline(0,2.5)(0,-2)
        \color{red}
        \pgfsetlinewidth{0.8pt}
        \pgfsetendarrow{\pgfarrowto}
        \pgfxyline(0,0)(1,0)
        \color{red}
        \pgfsetlinewidth{0.8pt}
        \pgfsetendarrow{\pgfarrowto}
        \pgfxyline(0,0)(-1,1)
        \color{red}
        \pgfsetlinewidth{0.8pt}
        \pgfsetendarrow{\pgfarrowto}
        \pgfxyline(0,0)(-2,-1)
        \color{black}
        \pgfputat{\pgfxy(0.8,1.2)}{\pgfbox[center,center]{\small $\Delta_{\tau_2\cap\tau_3}$}}
        \pgfputat{\pgfxy(0.6,-1)}{\pgfbox[center,center]{\small $\Delta_{\tau_1\cap\tau_2}$}}
        \pgfputat{\pgfxy(-1.4,0.4)}{\pgfbox[center,center]{\small $\Delta_{\tau_1\cap\tau_3}$}}
        \pgfputat{\pgfxy(1.4,0)}{\pgfbox[center,center]{\small{$(1,0)$}}}
        \pgfputat{\pgfxy(-1.2,1.2)}{\pgfbox[center,center]{\small{$(-1,1)$}}}
        \pgfputat{\pgfxy(-1.8,-1.2)}{\pgfbox[center,center]{\small{$(-v,-u)$}}}
\end{pgfpicture}
\end{multicols}
\caption{$\Gamma_\infty((y^{-1}+x)^u+y^v)$ and the conical partition of $\mathbb{R}^2$ induced by it.}
\end{figure}

Note that $\mathcal{C}=\mathbb{R}_+^2\setminus\{0\}$ and $\overline{\mathcal{C}}$ is trivial. By introducing the ray $\left(  1,1\right)
\in\mathbb{R}_{+}^{2}$, we get a nontrivial and simple fan $\mathcal{F}$ subordinated to $\Gamma_{\infty}$ and supported on $\mathbb{R}_{+}^{2}$.
\end{example}

In the classical case of a polynomial function $f$, the description of the set of generators of a simple fan subordinated to the Newton polyhedron of $f$, $\Gamma(f)$, is very simple and consists just of the normal vectors to the facets of $\Gamma(f)$ and a finite set of primitive vectors
corresponding to the extra rays induced by the subdivision into simple cones. In our approach,  the description of the set of generators is not so simple, see \cite{LeZu}*{Remark 3.5 (ii)}. When $n=2,$ we have
\begin{equation}\label{decomposition}
gen(\mathcal{F})=\mathfrak{D}(\mathcal{F})\cup\mathcal{E}(\mathcal{F})\cup\mathcal{E}^{\prime}(\mathcal{F}),
\end{equation}
where $\mathfrak{D}(\mathcal{F})\subset\mathfrak{D}(\Gamma_{\infty}),$
$\mathcal{E}(\mathcal{F})\subseteq\{e_{1},e_{2}\}$ and $\mathcal{E}^{\prime}(\mathcal{F})$ is the set of rays coming from the subdivision into simple cones. When $n\geq3$, the set $gen(\mathcal{F})$ becomes
\begin{equation*}
gen(\mathcal{F})=\mathfrak{D}(\mathcal{F})\cup\mathfrak{D}^\prime(\mathcal{F})\cup\mathcal{E}(\mathcal{F})\cup\mathcal{E}^{\prime}(\mathcal{F}),
\end{equation*}
where the set $\mathfrak{D}^\prime(\mathcal{F})$ consists of the rays coming from the intersection of $\mathbb{R}_{+}^{n}$ with some of the cones in the corresponding simplicial fan subordinated to $\Gamma_{\infty}$ and supported on  $\mathbb{R}^{n}$. The situation is illustrated in the following example.
\begin{example}\label{Ex 2}
Take $f(x,y,z)=x^{-1}+y^{-1}+z^{-1}(1+y)+{(xz)}^{-1}y\in K[x^{\pm 1},y^{\pm 1},z^{\pm 1}]$. Then the set of inward normal vectors of $\Gamma_\infty$ is  \[\mathfrak{D}(\Gamma_{\infty})=\{(1,1,1),(0,0,1),(0,-1,-1),(-1,-1,-2),(-1,0,0)\}.\] The cone generated by the vectors $(0,0,1),\ (1,1,1)$ and $(-1,0,0)$ belongs to the original simple fan supported on $\mathbb{R}^3$, but its intersection with $\mathbb{R}_+^3$ is a $3-$dimensional cone generated by the vectors $(0,0,1),\ (1,1,1)$ and $(0,1/2,1/2)$. The last vector is in neither set of (\ref{decomposition}). 
\end{example}
A finner description of the set of candidate poles of the local zeta functions for Laurent polynomials, and perhaps of our results (see Theorem \ref{Th2}), will require a characterization of the rays in $\mathfrak{D}^\prime(\mathcal{F})$, but so far this subject has scape from the attempts of the author. This matter  was not clarified enough in our previous work.

\subsection{Non--degeneracy Conditions}

Given $a\in \mathbb{R}_+^n$, we define the \textit{face
function} of $f(x)=\sum_{m\in S}c_{m}x^m$ with respect to $a$ as
\[
f_a(x)=\sum_{m\in S\cap F(a)}c_{m}x^m.
\]
 We set $T^{n}=T^{n}\left(  K\right)  :=\left\{
(x_{1},\ldots,x_{n})\in K^{n}\ ;\  x_{1}\cdots x_{n}\neq0\right\}  $, for the
$n$-dimensional torus considered as a $K$-analytic manifold. 
\begin{definition}\label{nondegerate}
Let $f(x)=\sum_{m\in S}c_{m}x^m\in K[x^{\pm 1}],$ be a non constant Laurent polynomial with
Newton polytope $\Gamma_{\infty}$. We say that $f$ is non--degenerate
with respect to  $a\in \mathbb{R}_+^n$, if  the system of equations
$$\left\{  f_{a}(x)=0,\nabla f_{a}(x)=0\right\}$$ has no
solutions in $T^{n}\left(  K\right)$. We say that $f$ is weakly non--degenerate with respect
to $\Gamma_{\infty}$ over $K$, if $f$ is non--degenerate with respect to any $a\in \mathbb{R}_+^n$.
\end{definition}

\begin{remark}
By allowing $a$ to take values on $\mathbb{R}^n$, we have that the property of being non-degenerate with respect to any $a\in\mathbb{R}^n$
is precisely the standard non- degeneracy condition of Khovanskii, see \cite{LeZu}*{Section 2.3}.
\end{remark}
\begin{example}
Let $f(x,y)$ as in Example \ref{ex. 1}. Then $f$ is degenerate with respect to $\Gamma_\infty$ but $f$ is weakly non-degenerate with respect to  $\Gamma_\infty$. 
\end{example}

Now we consider  $L\subseteq\mathbb{C}$ a number field with ring of
integers $\mathcal{R}_{L}$. Take a maximal ideal $\mathfrak{m}$ of $\mathcal{R}_{L}$, and assume that $\mathbb{F}_{q}$ is the residue field of
$\mathfrak{m}$. In addition, denote by $K:=K_{\mathfrak{m}}$ the completion of $L$ with respect to the $\mathfrak{m}$-adic valuation and denote by $\overline{f}$ the reduction modulo $\mathfrak{m}$ of $f$. 

\begin{definition}\label{nondegerateFq}
	A non constant Laurent polynomial $f(x)=\sum_{m\in S}c_{m}x^m\in K[x^{\pm 1}],$ is weakly non--degenerate with respect
	to $\Gamma_{\infty}$ over $\mathbb{F}_{q}$ if for any $a\in \mathbb{R}_+^n$, 
	the system of equations
	\begin{equation*}
	\{\overline{f_{a}}(x)=0,\ \nabla\overline{f_{a}}(x)=0\} 
	\end{equation*}
	has no solutions in $(\mathbb{F}_{q}^{\times})^{2}$.
\end{definition}

We may assume that a non--degenerated Laurent polynomial can be written as
\begin{equation}\label{quotient}
	f(x)=\frac{\widehat{f}(x)}{x^m},
\end{equation}
for some $m\in\mathbb{N}$ and $\widehat{f}(x)\in K[x_1,\ldots,x_n]$. As in the classical case, if $f(x)\in K[x^{\pm 1}]$ is a weakly non--degenerate Laurent polynomial  with respect to $\Gamma_{\infty}$ over $K$,  then for almost all $\mathfrak{m}$, i.e. for $q$ big enough, $f$ is weakly non--degenerate  with respect to $\Gamma_{\infty}$ over $\mathbb{F}_{q}.$ This fact follows easily from (\ref{quotient}) and the corresponding remarks in \cite{DeHoo}*{p.4}. 

\section{Local Zeta Functions for Laurent Polynomials}

In this section we give the definition and the main properties of the `twisted' local
zeta functions attached to a Laurent polynomial, of course we only summarize the results presented in \cite{LeZu}. 

Let $K$ be a $p-$adic field, i.e. $[K:\mathbb{Q}_{p}]<\infty$.  Let $R_{K}$\ be the
valuation ring of $K$, $P_{K}$ the maximal ideal of $R_{K}$, and $\overline
{K}=R_{K}/P_{K}$ \ the residue field of $K$, which is a finite field of lets say $q$ elements, i.e. $\overline{K}=\mathbb{F}_{q}$. Given $z\in
K$, $ord\left(  z\right)$  will denote its valuation and $\left\vert z\right\vert _{K}=q^{-ord\left(  z\right)
}$. We will fix a uniformizer $\mathfrak{p}$  of $P_{K}$ and define the angular component of $z$ by $ac\ z=z\mathfrak{p}^{-ord(z)}$.

We equip $K^{n}$ with the norm $\left\Vert (x_1,\ldots,x_n)
\right\Vert _{K}:=\max_{1\leq i\leq n}\left(  \left\vert x_i\right\vert _{K}\right)  $. Then $\left(  K^{n},\left\Vert \cdot
\right\Vert _{K}\right)  $ is a complete metric space and the metric topology
is equal to the product topology.

A quasicharacter of $K^{\times}$ is a continuous
homomorphism  $\omega:\ K^{\times}\to\mathbb{C}^{\times}$. The set of
quasicharacters  that we will denote by $\Omega\left(  K^{\times}\right)$ has an Abelian group structure. Now, to a given complex number $s$ we associate an element $\omega_{s}$ of $\Omega\left(  K^{\times
}\right)  $ by setting $\omega_{s}\left(  x\right)
=\left\vert x\right\vert _{K}^{s}$. When we choose that $\omega\left(  \mathfrak{p}\right)  =q^{-s}$ for every
$\omega\in\Omega\left(  K^{\times}\right)  $, we have that 
\begin{equation}\label{Characters}
\omega\left(  x\right)  =\omega_{s}\left(  x\right)  \chi\left(  ac\text{
}x\right),  
\end{equation}
where $\chi:=\omega\mid_{R_{K}^{\times}}.$ Equation (\ref{Characters}) shows that  $\Omega\left(
K^{\times}\right) \simeq\mathbb{C}/\left(  2\pi\sqrt{-1}/\ln
q\right)  \mathbb{\times}\left(  R_{K}^{\times}\right)  ^{\ast}$, where
$\left(  R_{K}^{\times}\right)  ^{\ast}$ is the group of characters of
$R_{K}^{\times}$; therefore $\Omega\left(  K^{\times}\right)  $ is a one dimensional
complex manifold. Now we note that $\sigma\left(  \omega\right)
:=\operatorname{Re}(s)$ depends only on $\omega$, and $\left\vert
\omega\left(  x\right)  \right\vert =\omega_{\sigma\left(  \omega\right)
}\left(  x\right)  $, thus it makes sense to define the following open subset of
$\Omega\left(  K^{\times}\right) , $ 
\[
\Omega_{(a,b)}\left(  K^{\times}\right)  =\left\{  \omega\in\Omega\left(
K^{\times}\right)  ;\sigma\left(  \omega\right)  \in(a,b)\subseteq\mathbb{R}\right\}  .
\]
We recall that a locally constant function on $K^{n}$ with compact support is
called a \textit{Bruhat-Schwartz function}, these functions form a
$\mathbb{C}$-vector space denoted as $S(K^{n})$. Given $\Phi\in S(K^{n}),\ \omega\in\Omega\left(  K^{\times}\right)$ and $f$ a Laurent polynomial, we define the 
local zeta function as
\[
Z_{\Phi}(\omega,f) =Z_{\Phi}(s,\chi,f)=\int\limits_{T^{n}\left(  K\right)  }
\Phi(x)\  \omega(f(x))\ \vert dx\vert,
\]
where $\vert dx\vert $ is the normalized Haar measure of $K^{n}$.

The convergence of $Z_{\Phi}(\omega,f)$ is not a straightforward
matter, this is an important difference with the classical case. We introduce here a few more notation in order to elucidate this assertion. 

For $a=\left(  a_{1},\ldots,a_{n}\right)  \in\mathbb{Z}^{n}\setminus
\left\{  0\right\}  $, we set $\left\Vert a\right\Vert =a_{1}+\ldots+a_{n}$ , and
\begin{equation}\label{P(A)}
\mathcal{P}(a)  :=\begin{cases}
 -\frac{\left\Vert a\right\Vert }{d\left(  a\right)  }+\frac{2\pi
 	\sqrt{-1}\mathbb{Z}}{d\left(  a\right)  \ln q}& \text{ if } d(a)\neq0,\\
 \emptyset & \text{ if }  d(a) =0.
\end{cases}
\end{equation}
Let $\mathcal{F}$ be the fixed simple fan subordinated to $\Gamma_{\infty}$
and supported in $\mathbb{R}_{+}^{n}$ as before.  Set
\begin{multline*}
A(\mathcal{F}):=\bigcup\limits_{\substack{a\in gen(\mathcal{F})\\d(a)\neq 0 }}
\left\{  \dfrac{\left\Vert a\right\Vert }{-d\left(  a\right)  }\ ;\ 
d\left(  a\right)  <0\right\},\quad B\left(  \mathcal{F}\right)  :=\bigcup\limits_{\substack{a\in gen(\mathcal{F})\\d(a)\neq 0 }}
\left\{  \dfrac{\left\Vert a\right\Vert }{-d\left(  a\right)  }\ ;\ d\left(  a\right)  >0\right\}  ,\\
\alpha:=\alpha(\mathcal{F})=\begin{cases}
\min_{\gamma\in A\left(  \mathcal{F}\right)}&\text{ if }  A\left(  \mathcal{F}\right)  \neq\emptyset\\
+\infty & \text{ if }  A\left(  \mathcal{F}\right)  =\emptyset,
\end{cases}
\quad
\text{and}\quad \beta:=\beta(\mathcal{F})=\max_{\gamma\in B\left(  \mathcal{F}\right)
\cup\left\{  -1\right\}  }\gamma.
\end{multline*}
\begin{theorem}[\protect{\cite{LeZu}*{Theorem 3.3}}] \label{Th1}
Let $f$ be a weakly non-degenerate Laurent polynomial with respect to
$\Gamma_{\infty}$ over $K$, and let $\mathcal{F}$ be a fixed simple and non-trivial fan subordinated to $\Gamma_{\infty}$. Then the following assertions hold:
\begin{enumerate}[(i)]
\item $Z\Phi(\omega,f)$ converges for $Re(s)\in(\beta,\alpha)$;
\item $Z_\Phi(\omega,f)$ has a
meromorphic continuation to the whole complex plane as a rational
function of $q^{-s}$, and the poles belong to the set
\[
\bigcup\limits_{\substack{a\in gen(\mathcal{F})\\d(a)\neq 0 }}
\mathcal{P}(a)\cup\left\{  -1+\frac{2\pi\sqrt{-1}\text{ }\mathbb{Z}}{\ln
q}\right\}.
\]
In addition, the multiplicity of any pole is least or equal to $n$.
\end{enumerate}
\end{theorem}

\begin{example}
Take $f(x,y) = (y^{-1}+x)^u+y^v\in K[x^{\pm1},y^{\pm1}]$, with arbitrary $u,v\in \mathbb{N}$, as in Example \ref{ex. 1}. Take $\Phi$ as the
characteristic function of $P_{K}^{2}$, and
$\omega=\omega_{s}$. Since $A(\mathcal{F})=1/u,\ B(\mathcal{F})=\emptyset$, we have that $\alpha=1/u$ and $\beta=-1$. So $Z_\Phi(\omega_s,f)$ will converge on the interval $(-1,1/u)$. But actually the integral does converge for $Re(s) < \frac{1}{n}$, as the following 
straightforward computation shows.
\begin{gather*}
Z_\Phi(\omega_s,f) =\int\limits_{(P_K\setminus\{0\})^2} |f(x,y)|_K^s\ |dxdy|=\sum_{a=1}^{\infty}\sum_{b=1}^{\infty}
\int\limits_{\mathfrak{p}^aR_K^\times\times\mathfrak{p}^bR_K^\times} |f(x,y)|_K^s\ |dxdy|\\
=\frac{\left(  1-q^{-1}\right)  q^{-2+us}}{1-q^{-1+us}}.
\end{gather*}
\end{example}

\section{Explicit Formulas}\label{Sect_EF}

The main result of the paper is an explicit formula for $Z_\Phi(\omega,f)$ when $f$ is a weakly non-degenerate Laurent polynomial, $\Phi$ is the characteristic function of $R_K^n$ and $\omega$ is an arbitrary element of $\Omega(K^{\times})$. Such type of formulas are known in the classical case of local zeta functions for polynomial functions: see \cite{DeHoo} for the case $K=\mathbb{Q}_p$ and $\chi=\chi_{triv}$, see \cite{Hoo} for the case $K=\mathbb{Q}_p$ and $\chi$ a non trivial character of $\mathbb{Z}_p^\times$ with conductor 1. The conductor $c_\chi$ of $\chi$ is defined as the smallest $c\in\mathbb{N}\setminus\{0\}$, such that $\chi$ is trivial on $1+p^c\mathbb{Z}_p$. Here we follow the same idea of \cite{Hoo}, but our results do not require any assumption about the size of $c_\chi$, this is due to the approach in the computation of the integrals appearing in the explicit formula for $Z_\Phi(\omega,f)$. The description of our approach is the subject of the following section.
\subsection{Stationary Phase Formula}

The stationary phase formula was introduced by Igusa with the aim of provide a tool for the explicit computation of some local zeta functions, the reader interested in the details may consult \cite{IBook}. 

We denote by $\bar{x}$ the image of $x\in R_K^n$ under the canonical homomorphism $R_K^n\to(R_K/P_K)^n\simeq \mathbb{F}_q^n$, i.e. $\bar{x}$ is the reduction of $x$ modulo $\mathfrak{p}$. Given a Laurent polynomial $f(x)\in R_K[x^{\pm 1}]$ such that not all its coefficients are in $P_K$, we denote by $\bar{f}$ the polynomial obtained by reducing the coefficients of $f$ mod $\mathfrak{p}$. We fix a lifting $L$ of $\mathbb{F}_q$ in $R_K$, thus we have that $L^n$ is mapped bijectively onto $\mathbb{F}_q^n$ by the aforementioned canonical homomorphism. Let  $\overline{D}$ be a subset of $\mathbb{F}_q^n$ and $D$ its preimage under the homomorphism. Let $S(f,D)$ denote the subset of $L^n$ mapped bijectively to the set of singular points of $\bar{f}$ lying on $\bar{D}$. We define also
\[\nu(\bar{f},D,\chi)=\begin{cases}
q^{-n}Card \left\{\bar{x}\in \bar{D}\ \mathbf{;}\ \bar{f}(\bar{x})\neq 0\right\} & \text{ if } \chi=\chi_{triv},\\
q^{-nc_\chi}\sum\limits_{\{x\in D\ \mathbf{;}\ \bar{f}(\bar{x})\neq 0\} \mod P_K^{c_\chi}}\chi(ac\ f(x)) & \text{ if } \chi\neq\chi_{triv},
\end{cases}\]
and 
\[\sigma(\bar{f},D,\chi)=\begin{cases}
q^{-n}Card \{\bar{x}\in \bar{D}\ \mathbf{;}\ \bar{x} \text{ is a non singular root of } \bar{f}\} & \text{ if } \chi=\chi_{triv},\\
0 & \text{ if } \chi\neq\chi_{triv}.
\end{cases}\]
We also set $Z_D(s,\chi,f)$ for the integral $\int\limits_{D}  \omega(f(x))\ \vert dx\vert=\int\limits_{D}  \chi(ac\ f(x))\ |f(x)|_K^s\ \vert dx\vert.$

\begin{lemma}\label{SPF}
	With the above notations, we have
	\[ Z_D(s,\chi,f)= \nu(\bar{f},D,\chi)+ \sigma(\bar{f},D,\chi)\frac{(1-q^{-1})q^{-s}}{1-q^{-1-s} }+\int\limits_{S(f,D)}  \omega(f(x))\ \vert dx\vert,\] where $Re(s)>0.$
\end{lemma}
\begin{proof}
	The proof given by Igusa in \cite{IBook} for the case of polynomial functions and trivial characters can be easily adapted to the case of arbitrary characters and Laurent polynomials, when $q$ is assumed big enough.
\end{proof}

\begin{theorem}\label{Th2} 
Let $f$ be a Laurent polynomial which is weakly non--degenerate with respect to $\Gamma_{\infty}$ over $\mathbb{F}_q$. Let $\mathcal{C}:=\{Int(\Delta_{\tau}\cap\mathbb{R}_{+}^{n})\ ;\ \Delta_{\tau}\cap\mathbb{R}_{+}^{n}\neq\emptyset\}\bigcup \cup_{i=1}^n\Delta_i$, be the simplicial partition of $\mathbb{R}_{+}^{n}\setminus\{0\}$ subordinated to $\Gamma_{\infty}$. Let $s$ be a complex number such that $\beta<Re(s)<\alpha$, and let $a_{1},\ldots,a_l,$ be the generators of a cone $\Delta\in\mathcal{C}$. Then
\[Z_{0}(s,\chi,f):=Z_{R_K^n}(s,\chi,f)=L_{\Gamma_\infty}(q^{-s},\chi)+\sum\limits_{\Delta\in\mathcal{C}}L_{\Delta}(q^{-s},\chi)S_{\Delta}(q^{-s}),\]
with
\[L_{\Delta}(q^{-s},\chi)=\nu(\overline{f_\Delta},(R_{K}^\times)^{n},\chi)+ \sigma(\overline{f_\Delta},(R_{K}^\times)^{n},\chi)\frac{(1-q^{-1})q^{-s}}{1-q^{-1-s} },\ \text{for any } \Delta\in\mathcal{C},\] and
\[S_{\Delta}\left(  q^{-s}\right)  =\frac{\left(\sum\limits_{h}q^{\left\Vert h\right\Vert +d\left(  h\right)  s}\right)  q^{^{-\sum_{j=1}^{l}\left(  \left\Vert a_{j}\right\Vert +d\left(a_{j}\right)  s\right)  }}}{\prod\limits_{j=1}^{l}\left(  1-q^{-\left\Vert a_{j}\right\Vert -d\left(  a_{j}\right)  s}\right)},\]
where $h$ runs through the elements of the set
\[\mathbb{Z}^{n}\cap\left\{\sum\limits_{j=1}^{l}\lambda_{j}a_{j}\ ;\ 0\leq\lambda_{j}<1\text{ for }j=1,\ldots,l\right\}.\]
\end{theorem}
\begin{proof}
First note that $\mathbb{R}_{+}^{n}=\{0\}\cup\bigsqcup_{\Delta\in\mathcal{C}}\Delta.$ Then
\begin{gather*}
Z_{0}(s,\chi,f)=\sum\limits_{k\in\mathbb{N}^n}\int\limits_{\substack{x\in (R_{K}\setminus\{0\})^n\\ ord (x)=k}} \omega(f(x))\  |dx|\\
=\int\limits_{(R_{K}^\times)^{n}}\omega(f(x))\ \vert dx\vert
+\sum\limits_{\Delta\in\mathcal{C}}\ \sum\limits_{k\in\Delta\cap\mathbb{N}^n}\ \int\limits_{\substack{x\in (R_{K})^n\\ ord (x)=k}}\omega(f(x))\ \vert dx\vert.
\end{gather*}
Now we change variables as 
\begin{equation}\label{change}
x_i=\mathfrak{p}^{k_i}u_i,\text{ with } u_i\in(R_{K}^\times)^{n}, \text{ then }|dx|=q^{-||k||} \text{ and } x^m=\mathfrak{p}^{\langle k,m\rangle}u^m.
\end{equation}
We fix $\Delta\in\mathcal{C}$ and note that for any vector $a\in\Delta,\ f_a(x)$ is the same, so we will use $f_\Delta(x)$ instead of $f_a(x)$. Now we have 
\begin{equation}\label{change1}
f(x)=f_\Delta(x)+\sum\limits_{m\in S\setminus F(a)}c_{m}x^m, \text{ for every }a\in\Delta,
\end{equation}
and $\langle k,m \rangle=d(k)$ for any $m\in F(a)$ whereas 
$\langle k,m \rangle>d(k)$ for $m\in S\setminus F(a)$. From (\ref{change}) and (\ref{change1}) we get
\[f(x)=\mathfrak{p}^{d(k)}f_\Delta(u)+\mathfrak{p}^{d(k)+l}\tilde{f}_{\Delta,k}(u),\]where $l>0$ and 
$\tilde{f}_{\Delta,k}(u)\in R_K^\times[u]$. 
Therefore
\begin{gather*}
Z_{0}(s,\chi,f)=L_{\Gamma_\infty}(q^{-s},\chi)\\
+\sum\limits_{\Delta\in\mathcal{C}}\ \sum\limits_{k\in\Delta\cap\mathbb{N}^n}q^{-||k||-d(k)s} \int\limits_{(R_{K}^\times)^n} \omega(f_\Delta(u)+\mathfrak{p}^{l}\tilde{f}_{\Delta,k}(u))\ |du|,
\end{gather*}
where  $L_{\Gamma_\infty}(q^{-s},\chi)=\int\limits_{(R_{K}^\times)^{n}}\omega( f(x))\ \vert dx\vert$. Now, since $f$ is non degenerate with respect to $a$ (for any $a\in\Delta$) we have that $\overline{f_\Delta}$ has no singular points on $(\mathbb{F}_{q}^\times)^{n},$ thus Lemma \ref{SPF} implies
\begin{equation*}
Z_{0}(s,\chi,f)=L_{\Gamma_\infty}(q^{-s},\chi)
+\sum\limits_{\Delta\in\mathcal{C}}\ L_\Delta(q^{-s},\chi)\ \sum\limits_{k\in\Delta\cap\mathbb{N}^n}q^{-||k||-d(k)s}.
\end{equation*}
Finally, it is a well known fact that $\sum\limits_{k\in\Delta\cap\mathbb{N}^n}q^{-||k||-d(k)s}$ equals $S_\Delta$, see e.g. \cite{DeHoo}.
\end{proof}

\begin{example}
\label{Example2}Set $g(x,y)=x^{-3}+y^{-2}+y^{4}\in {R_K}\left[
x^{\pm 1},y^{\pm 1}\right]  $. We choose $\Phi$ to be the characteristic function of $R_K^2$ and $\chi=\chi_{triv}.$ Furthermore, we will assume that $char(K)$ is different from 2 and 3. Note that $g(x,y)$ is a non-degenerate  polynomial with respect to
$\Gamma_{\infty}\left(  g\right)  $ over $\mathbb{F}_{q}$, for $q$ big enough. 
\begin{figure}[ht]
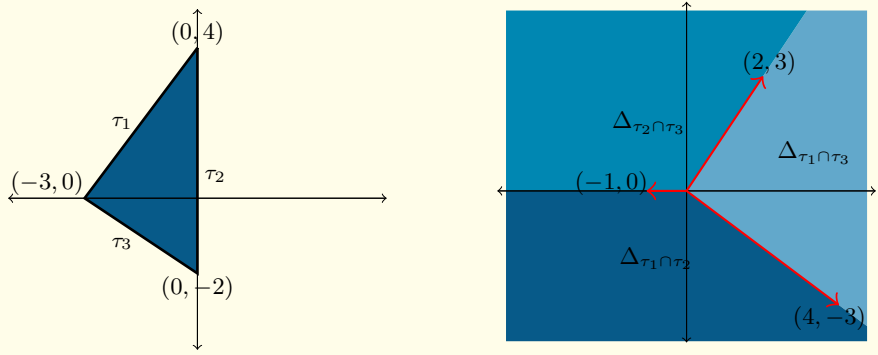
\label{Fig2}	
\begin{multicols}{2}
\begin{pgfpicture}{-2.5cm}{-2cm}{2.5cm}{2.5cm}
        \color{darkblue}
        \pgfmoveto{\pgfxy(0,2)}
        \pgflineto{\pgfxy(0,-1)}
        \pgflineto{\pgfxy(-1.5,0)}
        \pgflineto{\pgfxy(0,2)}
        \pgfclosepath
        \pgffill
        \pgfsetlinewidth{1pt}
        \color{black}
        \pgfmoveto{\pgfxy(0,2)}
        \pgflineto{\pgfxy(0,-1)}
        \pgflineto{\pgfxy(-1.5,0)}
        \pgflineto{\pgfxy(0,2)}
        \pgfstroke
        \pgfsetlinewidth{0.4pt}
        \pgfsetendarrow{\pgfarrowto}
        \pgfxyline(-2.5,0)(2.5,0)
        \pgfxyline(0,-2)(0,2.5)
        \pgfxyline(2.5,0)(-2.5,0)
        \pgfxyline(0,2.5)(0,-2)
                
        \pgfputat{\pgfxy(-1,1)}{\pgfbox[center,center]{\footnotesize $\tau_1$}}
        \pgfputat{\pgfxy(0.1,0.3)}{\pgfbox[left,center]{\footnotesize $\tau_2$}}
        \pgfputat{\pgfxy(-1,-0.6)}{\pgfbox[center,center]{\footnotesize $\tau_3$}}
        \pgfputat{\pgfxy(0,2.2)}{\pgfbox[center,center]{\small{$(0,4)$}}}
        \pgfputat{\pgfxy(-2,0.2)}{\pgfbox[center,center]{\small{$(-3,0)$}}}
        \pgfputat{\pgfxy(0,-1.2)}{\pgfbox[center,center]{\small{$(0,-2)$}}}
        \end{pgfpicture}

\columnbreak
\begin{pgfpicture}{-2.4cm}{-2cm}{2.4cm}{2.4cm}
        \color{lightblue}
        \pgfmoveto{\pgfxy(2.4,-1.8)}
        \pgflineto{\pgfxy(2.4,2.4)}
        \pgflineto{\pgfxy(1.6,2.4)}
        \pgflineto{\pgfxy(0,0)}
        \pgflineto{\pgfxy(2.4,-1.8)}
        \pgfclosepath
        \pgffill
        \color{midblue}
        \pgfmoveto{\pgfxy(1.6,2.4)}
        \pgflineto{\pgfxy(-2.4,2.4)}
        \pgflineto{\pgfxy(-2.4,0)}
        \pgflineto{\pgfxy(0,0)}
        \pgflineto{\pgfxy(1.6,2.4)}
        \pgfclosepath
        \pgffill
        \color{darkblue}
        \pgfmoveto{\pgfxy(0,0)}
        \pgflineto{\pgfxy(-2.4,0)}
        \pgflineto{\pgfxy(-2.4,-2)}
        \pgflineto{\pgfxy(2.4,-2)}
        \pgflineto{\pgfxy(2.4,-1.8)}
        \pgflineto{\pgfxy(0,0)}
        \pgfclosepath
        \pgffill
        \color{black}
        \pgfsetlinewidth{0.4pt}
        \pgfsetendarrow{\pgfarrowto}
        \pgfxyline(-2.5,0)(2.5,0)
        \pgfxyline(0,-2)(0,2.5)
        \pgfxyline(2.5,0)(-2.5,0)
        \pgfxyline(0,2.5)(0,-2)
        \color{red}
        \pgfsetlinewidth{0.8pt}
        \pgfsetendarrow{\pgfarrowto}
        \pgfxyline(0,0)(1,1.5)
        \color{red}
        \pgfsetlinewidth{0.8pt}
        \pgfsetendarrow{\pgfarrowto}
        \pgfxyline(0,0)(-0.5,0)
        \color{red}
        \pgfsetlinewidth{0.8pt}
        \pgfsetendarrow{\pgfarrowto}
        \pgfxyline(0,0)(2,-1.5)
        \color{black}
        \pgfputat{\pgfxy(-0.5,0.9)}{\pgfbox[center,center]{\small $\Delta_{\tau_2\cap\tau_3}$}}
        \pgfputat{\pgfxy(-0.4,-0.9)}{\pgfbox[center,center]{\small $\Delta_{\tau_1\cap\tau_2}$}}
        \pgfputat{\pgfxy(1.7,0.5)}{\pgfbox[center,center]{\small $\Delta_{\tau_1\cap\tau_3}$}}
        \pgfputat{\pgfxy(-1,0.1)}{\pgfbox[center,center]{\small{$(-1,0)$}}}
        \pgfputat{\pgfxy(1.1,1.7)}{\pgfbox[center,center]{\small{$(2,3)$}}}
        \pgfputat{\pgfxy(1.9,-1.7)}{\pgfbox[center,center]{\small{$(4,-3)$}}}
        \end{pgfpicture}
\end{multicols}
\caption{$\Gamma_\infty(x^{-3}+y^{-2}+y^{4})$ and the conical partition of $\mathbb{R}^2$ induced by it.}
\end{figure}

The vectors $\left\{\left(  1,0\right)  ,\left(  2,3\right)  ,\left(  0,1\right)  \right\}  $ are
the edges of a non-trivial simplicial polyhedral subdivision of $\mathbb{R}%
_{+}^{2}$ subordinated to $\Gamma_{\infty}\left(  g\right)  $. Therefore
\[
L_{\Gamma_{\infty}}\left(  q^{-s},\chi_{triv}\right)  =q^{-2}\left\{  (q-1)^{2}+N\left(  \frac{q^{-s}-1}{1-q^{-s-1}}\right)  \right\}  ,
\]
with $N=\left\{  \left(  x,y\right)  \in\left(
\mathbb{F}_{q}^{\times}\right)  ^{2}\mathbf{;}\ x^{-3}+y^{-2}+y^{4}=0\right\}
$. The remaining data for the explicit formula for $Z_0(s, x^{-3}+y^{-2}+y^{4},\chi_{triv})$ are in Table \ref{Tab1}.
\begin{table}\label{Tab1}
\centering
\begin{tabular}
[c]{|c|c|c|c|}\hline
$%
\begin{array}
[c]{c}%
\text{Cone}\\
\text{Generators}%
\end{array}
$ & $\begin{array}
[c]{c}%
\text{Face}%
\end{array}$
 & $L_{\Delta}\left(  q^{-s},\chi_{triv}\right)  $ & $S_{\Delta}\left(
q^{-s}\right)  $\\\hline
$\left(  0,1\right)  $ & $\tau_2\cap\tau_3  $ & $(1-q^{-1})^{2}$ & $\dfrac{q^{-1+2s}}{1-q^{-1+2s}}$\\\hline
$\left(  2,3\right)  $ & $\tau_3$ & $%
\begin{array}
[c]{c}%
q^{-2}(q-1)\{(q-1)+\\
\frac{q^{-s}-1}{1-q^{-s-1}}  \}
\end{array}
$ & $\dfrac{q^{-5+6s}}{1-q^{-5+6s}}$\\\hline
$(1,0)$ & $\tau_1\cap\tau_3  $ & $\left(  1-q^{-1}%
\right)  ^{2}$ & $\dfrac{q^{-1+3s}}{1-q^{-1+3s}}$\\\hline
$\left(  0,1\right)  ,\left(  2,3\right)  $ & $\tau_2\cap\tau_3  $ & $\left(  1-q^{-1}\right)  ^{2}$ & $\dfrac{(1+q^{3-4s})q^{-6+8s}}{(1-q^{-1+2s})(1-q^{-5+6s})}$\\\hline
$\left(  2,3\right)  ,\left(  1,0\right)  $ & $\tau_1\cap\tau_3  $ & $\left(  1-q^{-1}\right)  ^{2}$ & $\dfrac{(1+q^{2-3s}+q^{4-6s})q^{-6+9s}}{(1-q^{-1+3s})(1-q^{-5+6s})}$
\\\hline
\end{tabular}
\caption{$g(x,y)=x^{-3}+y^{-2}+y^{4}$}
\end{table}
It follows that  the real parts of the poles of the local zeta function belong
to $\left\{  \frac{1}{2},\frac{1}{3},\frac{5}{6},-1\right\}  $.  Note that only the poles $\frac{5}{6}%
+\frac{2\pi\sqrt{-1}\mathbb{Z}}{6\log q}$ come from the equation of a
supporting plane, more precisely from the face $\tau_3$. Finally, note that
$\beta=-1$ and $\alpha=\frac{1}{3}$, and that this last datum does not come from the equations of the supporting planes
of $\Gamma_{\infty}\left(g\right)$.
\end{example}

\end{document}